\documentclass[a4paper]{amsart}
\usepackage{amsthm,amsmath,amssymb,mathrsfs}
\usepackage[latin1]{inputenc}
\usepackage{mathrsfs}
\usepackage{mathtools}
\usepackage{url}

\usepackage{tikz} 
\usetikzlibrary{matrix,arrows,calc,decorations.pathreplacing,fit}
\usepackage{tikz-cd}

\numberwithin{equation}{section}

\newtheorem{theorem}[equation]{Theorem}
\newtheorem{lemma}[equation]{Lemma}
\newtheorem{lemdef}[equation]{Lemma/Definition}
\newtheorem{corollary}[equation]{Corollary}
\newtheorem{proposition}[equation]{Proposition}
\newtheorem{conjecture}[equation]{Conjecture}
\theoremstyle{definition}
\newtheorem{definition}[equation]{Definition}

\newtheorem{notation}[equation]{Notation}
\theoremstyle{remark}
\newtheorem{remark}[equation]{Remark}
\newtheorem{example}[equation]{Example}

\DeclareMathOperator{\Cent}{Cent}

\DeclareMathOperator{\cha}{char} 

\DeclareMathOperator{\defect}{def}

\DeclareMathOperator{\image}{im}

\DeclareMathOperator{\Ind}{c-ind}

\DeclareMathOperator{\Gal}{Gal}
\DeclareMathOperator{\GL}{GL}
\DeclareMathOperator{\GSp}{GSp}

\DeclareMathOperator{\Norm}{Norm}

\DeclareMathOperator{\PGL}{PGL}
\DeclareMathOperator{\pr}{pr}

\DeclareMathOperator{\Spec}{Spec}

\DeclareMathOperator{\Res}{Res}

\DeclareMathOperator{\val}{val}

\def \AA {\mathbb{A}}

\def \GG {\mathbb{G}}

\def \NN {\mathbb{N}}

\def \QQ {\mathbb{Q}}

\def \ZZ {\mathbb{Z}}

\def \Acal {\mathcal{A}}

\def \Ecal {\mathcal{E}}

\def \Ocal {\mathcal{O}}

\def \Dscr {\mathscr{D}}

\def \Mscr {\mathscr{M}}

\def \Bhat {\hat{B}}

\def \Ghat {\hat{G}}
\def \Hhat {\hat{H}}

\def \Qhat {\hat{Q}}

\def \Shat {\hat{S}}
\def \That {\hat{T}}

\def \fbar {\bar{f}}

\def \hbar {\bar{h}}

\def \sbar {\bar{s}}

\def \Htilde {\tilde{H}}

\def \Ttilde {\tilde{T}}

\def \Wtilde {\tilde{W}}

\def \btilde {\tilde{b}}

\def \wbf {\mathbf{w}}

\def \mono  {\hookrightarrow}

\def \epi   {\twoheadrightarrow}

\def \isom  {\stackrel{\sim}{\rightarrow}}

\def \bij   {\stackrel{1:1}{\rightarrow}}

\newcommand{\pot}[1]{ [\hspace{-0,17em}[ {#1} ]\hspace{-0,17em}] }
\newcommand{\rpot}[1]{ (\hspace{-0,23em}( {#1} )\hspace{-0,23em}) }

\newcommand{\bigslant}[2]{{\raisebox{.2em}{$#1$}\hspace{-.3em}\left/ \hspace{-.2em}\raisebox{-.2em}{$#2$}\right.}}

\newcommand{\restr}[2]{{#1}\raise-.5ex\hbox{\ensuremath|}_{#2}}

\def \dom {{\rm dom}}

\newcounter{subenvcounter}
\newenvironment{subenv}{%
 \begin{list}
  {\em (\arabic{subenvcounter})}
  {\setlength{\leftmargin}{20pt}
   \setlength{\rightmargin}{0pt}
   \setlength{\itemindent}{0pt}
   \setlength{\labelsep}{5pt}
   \setlength{\labelwidth}{13pt}
   \setlength{\listparindent}{\parindent}
   \setlength{\parsep}{0pt}
   \setlength{\itemsep}{0pt}
   \setlength{\topsep}{-\parskip}
   \usecounter{subenvcounter}}}
  {\end{list}}

\newcounter{asslistcounter}
\newenvironment{assertionlist}{
 \begin{list}
  {\upshape (\alph{asslistcounter})}
  {\setlength{\leftmargin}{18pt}
   \setlength{\rightmargin}{0pt}
   \setlength{\itemindent}{0pt}
   \setlength{\labelsep}{5pt}
   \setlength{\labelwidth}{13pt}
   \setlength{\listparindent}{\parindent}
   \setlength{\parsep}{0pt}
   \setlength{\itemsep}{0pt}
   \setlength{\topsep}{-.5\parskip}
   \usecounter{asslistcounter}}}
  {\end{list}}

\DeclareMathOperator{\type}{type}

\DeclareMathOperator{\cotype}{cotype}

\def \ad {\mathrm{ad}}

\def \dom {\mathrm{dom}}

\def \top {\mathrm{top}}
\def \pr {\mathrm{pr}}

\def \Y {Y}

\def \Grass {\mathcal{G}r}

\begin{document}

\begin{title}
{Irreducible components of minuscule affine Deligne-Lusztig varieties}
\end{title}
\author{Paul Hamacher and Eva Viehmann}
\address{Technische Universit\"at M\"unchen\\Fakult\"at f\"ur Mathematik - M11 \\ Boltzmannstr. 3\\85748 Garching bei M\"unchen\\Germany}
\email{hamacher@ma.tum.de, viehmann@ma.tum.de}
\date{}
\thanks{The authors were partially supported by ERC starting grant 277889 ``Moduli spaces of local $G$-shtukas''.}

 \setcounter{tocdepth}{1}

\begin{abstract}{We examine the set of $J_b(F)$-orbits in the set of irreducible components of affine Deligne-Lusztig varieties for a hyperspecial subgroup and minuscule coweight $\mu$. Our description implies in particular that its number of elements is bounded by the dimension of a suitable weight space in the Weyl module associated with $\mu$ of the dual group.}
\end{abstract}

\maketitle

 \tableofcontents

\section{Introduction}\label{intro}

Let $F$ be a finite extension of $\mathbb{Q}_p$ or $\mathbb{F}_p(\!(t)\!)$ and $\Gamma$ its absolute Galois group. We denote by $\Ocal_F$ and $k_F\cong \mathbb {F}_q$ its ring of integers and its residue field, and by $\epsilon$ a fixed uniformiser. Let $L$ denote the completion of the maximal unramified extension of $F$, and $\Ocal_{L}$ its ring of integers. Its residue field is an algebraic closure $k$ of $k_F$. We denote by $\sigma$ the Frobenius of $L$ over $F$ and of $k$ over $k_F$.

Let $G$ be a reductive group scheme over $\Ocal_F$ and denote $K =G(\Ocal_L)$. Then $G_F$ is automatically unramified. We fix $S \subset T \subset B \subset G$, where $S$ is a maximal split torus, $T$ a maximal torus, and $B$ a Borel subgroup of $G$. Let $W$ be the absolute Weyl group of $G$. There exist $k_F$-ind schemes called the loop group $LG$, the positive loop group $L^+G$ and the affine Grassmannian $\Grass_G \coloneqq LG/L^+G$ of $G$ whose $k$-valued points are canonically identified with $G(L)$, $K=G(\Ocal_L)$ and $G(L)/G(\Ocal_L)$, respectively (compare \cite{PR08} resp.~\cite{zhu} and \cite{BS}).

Let $\mu\in X_*(T)_{\dom}$ and let $b\in G(L)$. Then the affine Deligne-Lusztig variety associated with $b$ and $\mu$ is the reduced subscheme $X_{\mu}(b)$ of  $\Grass_G$  whose $k$-valued points are 
$$X_{\mu}(b)(k)=\{g\in G(L)/K \mid g^{-1}b\sigma(g)\in K\mu(\epsilon) K\}.$$
Let $X_{\preceq\mu}(b)=\bigcup_{\mu'\preceq \mu} X_{\mu'}(b)$ where $\mu'\preceq\mu$ if $\mu-\mu'$ is a non-negative integral linear combination of positive coroots. It is closed in the affine Grassmannian and called the closed affine Deligne-Lusztig variety. For minuscule $\mu$ (the case we are mainly interested in for this paper) it agrees with $X_{\mu}(b)$. 

Notice that up to isomorphism, both affine Deligne-Lusztig varieties depend only on the $G(L)$-$\sigma$-conjugacy class $[b]\in B(G)$ of $b$. An affine Deligne-Lusztig variety $X_{\mu}(b)$ or $X_{\preceq \mu}(b)$ is non-empty if and only if $[b]\in B(G,\mu)$, a finite subset of $B(G)$. The following basic assertion seems to be well-known, but we could not find a reference in the literature.

\begin{lemma}\label{lemfintype}
The scheme $X_{\mu}(b)$ is locally of finite type in the equal characteristic case and locally of perfectly finite type in the case of unequal characteristic.
\end{lemma} 
\begin{proof}
The proof of this is the same as the corresponding part of the analogous statement for moduli spaces of local $G$-shtukas, compare the proof of Theorem 6.3 in \cite{HV11} (where only the first half of p. 113 of loc.~cit. is needed). In that proof, the case of equal characteristic and split $G$ is considered. However, the general statement follows from the same proof.
\end{proof}
Notice that in general $X_{\mu}(b)$ is not quasi-compact since it may have infinitely many irreducible components. It is conjectured to be equidimensional, but this has not been proven in full generality yet. In Section~\ref{secequidim} we give an overview about the cases where equidimensionality has been proven. In the case of $\mu$ minuscule, which we are primarily interested in here, there are only a few exceptional cases where this is not yet known.

\begin{definition}\label{def1}
For a finite-dimensional $k$-scheme $X$ we denote by $\Sigma(X)$ the set of irreducible components of $X$ and by $\Sigma^\top(X) \subset \Sigma(X)$ the subset of those irreducible components which are top-dimensional. 
\end{definition}

The affine Deligne-Lusztig varieties $X_{\mu}(b)$ and $X_{\preceq \mu}(b)$ carry a natural action (by left multiplication) by the group
$$J_b(F)=\{g\in G(L)\mid g^{-1}b\sigma(g)=b\}.$$
This action induces an action of $J_b(F)$ on the set of irreducible components.

A complete description of the set of orbits was previously only known for the groups $\GL_n$ and $\GSp_{2n}$ and minuscule $\mu$ where the action is transitive (\cite{modpdiv},\cite{polpdiv}), and for some other particular cases, see for example \cite{vollaardwedhorn} for a particular family of unitary groups and minuscule $\mu$.

To describe the (conjectured) number of orbits, denote by $\Ghat$ the dual group of $G$ in the sense of Deligne and Lusztig. That is, $\Ghat$ is the reductive group scheme over $\Ocal_F$ that contains a Borel subgroup $\Bhat$ with maximal torus $\That$ and maximal split torus $\Shat$ such that there exists an Galois equivariant isomorphism $X^\ast(\That) \cong X_\ast(T)$ identifying simple coroots of $\That$ with simple roots of $T$. For any $\mu \in X_\ast(T)_\dom = X^\ast(\That)_\dom$ we denote by $V_\mu$ the associated Weyl module of $\Ghat_{\Ocal_L}$.

In the following we use an element $\lambda_G(b) \in X^\ast(\That^\Gamma)$ that we define in Section \ref{secdeflambda}. Its  restriction $\lambda$ to $\Shat$ can be seen as a `best integral approximation' of the Newton point $\nu_b$ of $[b]$, while its precise value in $X^\ast(\That^\Gamma)$ will depend on the Kottwitz point $\kappa_G(b)$. We choose a lift $\tilde\lambda \in X_\ast(T)$.

\begin{conjecture}[Chen, Zhu] \label{conj main}
 There exists a canonical bijection between $J_b(F)\backslash \Sigma(X_{\mu}(b))$ and the basis of $V_\mu(\lambda_G(b))$ constructed by Mirkovic and Vilonen in \cite{MV07}, where $V_{\mu}(\lambda_G(b))$ denotes the $\lambda_G(b)$-weight space (for the action of $\That^\Gamma$) of $V_\mu$.
\end{conjecture}

In this paper, we describe the set $J_b(F)\backslash \Sigma^\top(X_{\mu}(b))$ for minuscule $\mu$. Our main result, Theorem \ref{thmnewmain}, implies in particular the following theorem.

\begin{theorem} \label{thm_main}
 Let $\mu\in X_*(T)_{\dom}$ be minuscule, $b\in [b] \in B(G,\mu)$, and $\tilde\lambda\in X_\ast(T)$ an associated element as in Section \ref{secdeflambda}. There exists a canonical surjective map
 \[
  \phi\colon   W.\mu \cap [\tilde\lambda+(1-\sigma)X_*(T)] \epi J_b(F)\backslash \Sigma^\top(X_{\mu}(b)).
 \]
 Moreover, this map is a bijection in the following cases.
 \begin{subenv}
  \item $G$ is split.
  \item $[b] \cap \Cent_G(\nu_b)$ is a union of superbasic $\sigma$-conjugacy classes in $\Cent_G(\nu_b)$.
 \end{subenv}

\end{theorem}

\begin{remark}\begin{assertionlist}

\item Let us explain how the theorem is a special case of the conjecture. Since $\mu$ is minuscule, we have for any $\tilde\mu \in X_\ast(T)$
  \[
   \dim V_\mu(\tilde\mu) = \left\{ \begin{array}{ll}
                               1 & \textnormal{if } \tilde{\mu} \in W.\mu \\
                               0 & \textnormal{otherwise}
                              \end{array} \right. 
  \]
where now $V_\mu(\tilde\mu)$ denotes the $\tilde\mu$-weight space for the action of $\hat T$. Thus indeed we obtain a bijection between the Mirkovic-Vilonen basis of $V_\mu(\lambda)$ and $ W.\mu \cap [\tilde\lambda+(1-\sigma)X_*(T)] $.

\item We can replace the weight space $V_\mu(\lambda_G(b))$ by the weight space $V_\mu(\lambda)$ for the action of $\Shat$ in Conjecture~\ref{conj main}. A priori one might expect the second space to be bigger; the equality is a consequence of the relation between $\lambda$ and the Kottwitz point $\kappa_G(b)$ (see Remark~\ref{rmk diag vs torus} for details).

\item An analogous formula has first been shown by Xiao and Zhu \cite{XZ} for $[b]$ such that the $F$-ranks of $J_b$ and $G$ coincide. In this case one can simply choose $\lambda=\nu_b$, the Newton point of $[b]$. It was then observed by Chen and Zhu (in oral communication) that an expression similar to the above should give $|J_b(F)\backslash \Sigma(X_{\mu}(b))|$ also for general $[b]$, and all $\mu$.

\item In particular, Theorem \ref{thm_main} and Theorem \ref{thmnewmain} apply to all cases that correspond to Newton strata in Shimura varieties of Hodge type.
\end{assertionlist}
\end{remark}

In the case where $b$ is superbasic, we prove the following stronger result, which was conjectured in \cite{hamacher15a}. For the ordering $\leq$ compare Section \ref{sec21}.

\begin{proposition} \label{prop superbasic}
 Assume $b \in G(L)$ is superbasic. There exists a decomposition into disjoint $J_b(F)$-stable locally closed subschemes
 \[
  X_\mu(b) = \bigcup_{\mathclap{\tilde\mu \in W.\mu \atop \restr{\tilde\mu}{\Shat} \leq \nu_b}} C_{\tilde\mu}
 \]
 such that $C_{\tilde{\mu}}$ intersected with any connected component of $\Grass_G$ is universally homeo\-morphic to 
an affine space. These affine spaces are of dimension $d(\tilde\mu) \coloneqq \sum \lfloor \langle  \tilde\mu - \mu_{{\rm adom}}, \hat\omega_F \rangle \rfloor$ where we take the sum over all relative fundamental coweights $\hat\omega_F$ of $\Ghat$ and where $\mu_{{\rm adom}}$ denotes the anti-dominant representative in the Weyl group orbit of $\mu$.
\end{proposition}

Note that varying $b$ within $[b]$ only changes $X_{\mu}(b)$ by an isomorphism. For suitably chosen $b\in [b]$, the connected components of $C_{\tilde\mu}$ are precisely the intersections of $X_\mu(b)$ with some Iwahori-orbit on $\Grass_G$ (see \cite[Section 3]{cvstrata}). Since the latter form a stratification on $\Grass_G$, we can apply the localisation long exact sequence to calculate the cohomology of $X_\mu(b)$. For example for the constant sheaf one obtains the following result.
\begin{corollary}
 Assume $b \in G(L)$ is superbasic and  denote by $J_b(F)^0$ the (unique) parahoric subgroup of $J_b(F)$. Then the $J_b(F)$-equivariant cohomology of $X_\mu(b)$ (for $\ell\neq p$) is given by
 \begin{align*}
  H_c^{2i+1}(X_\mu(b),\QQ_{\ell}) &= 0 \\
  H_c^{2i}(X_\mu(b),\QQ_{\ell}) &= \Ind_{J_b(F)^0}^{J_b(F)} V_i 
 \end{align*}
 where $V_i$ is a diagonalisable $J_b(F)^0$-representation with coefficients in $\QQ_{\ell}$ and of dimension $\#\{\tilde\mu \in W.\mu \mid d(\tilde\mu) = i\}$.
\end{corollary}

\noindent{\it Acknowledgement.} We thank Miaofen Chen and Xinwen Zhu for helpful conversations and in particular for sharing their conjecture describing the $J_b(F)$-orbits of irreducible components in terms of $V_{\mu}(\lambda)$.

\section{Definition of $\lambda$}\label{secdeflambda}

 We associate with every $\sigma$-conjugacy class $[b]$ a not necessarily dominant coinvariant $\lambda_G(b) \in X^\ast(\That)_\Gamma$ which lifts the Kottwitz point of $b$ and at the same time is a `best approximation' of the Newton point (in a sense to be made precise below). In the split case it is closely connected to the notion of $\sigma$-straight elements in the extended affine Weyl group of $G$.

 \subsection{Invariants of $\sigma$-conjugacy classes}\label{sec21}

 By work of Kottwitz \cite{kottwitz85}, a $\sigma$-conjugacy class $[b] \in B(G)$ is uniquely determined by two invariants - the Newton point $\nu_G(b) \in X_\ast(S)_{\QQ,\dom}$ and the Kottwitz point $\kappa_G(b) \in \pi_1(G)_\Gamma$. Here $\pi_1(G)$ denotes Borovoi's fundamental group, i.e.\ the quotient of $X_\ast(T)$ by its coroot lattice. We also consider the Kottwitz homomorphism $w_G$ as in \cite{Kottwitz85}. Let $w\colon  X_\ast(T) \epi \pi_1(G)$ denote the canonical projection. By the Cartan decomposition $G(L)=\coprod_{\mu\in X_*(T)_{\dom}}K\mu(\epsilon)K$, and we extend $w$ to a map $w_G: G(L)\rightarrow \pi_1(G)$ mapping $K\mu(\epsilon)K$ to $w(\mu)$. Then for every $b\in G(L)$ the projection of $w_G(b)$ to $\pi_1(G)_{\Gamma}$ coincides with $\kappa_G(b)$.

 We define a partial order $\preceq$ on $X^\ast(\That)$ such that $\mu' \preceq \mu$ holds iff $\mu - \mu'$ is a linear combination of positive roots with non-negative, integral coefficients. Since the set of positive roots is preserved by the Galois action, this descends to a partial order on $X^\ast(\That)_\Gamma$. Similarly, we define its rational analogue $\leq$ on $X^\ast(T)_\QQ$ such that $\mu \leq \mu'$ holds iff $\mu - \mu'$ is a linear combination of positive roots with non-negative, rational coefficients. By the same argument as above this order descends to $X^\ast(\That)_{\QQ,\Gamma} = X^\ast(\Shat)$.
 
 \begin{lemdef}
Let $b \in G(L)$. Then the set 
$$ \{ \tilde\lambda \in X^\ast(\That)_\Gamma \mid w(\tilde\lambda) = \kappa_G(b), \restr{\tilde\lambda}{\Shat} \leq
 \nu_G(b) \}$$
 has a unique maximum $\lambda_G(b)$ characterised by the property that $w(\lambda_G(b)) = \kappa_G(b)$ and that for every relative fundamental coweight $\omega_{\Ghat,F}^\vee$ of $\Ghat$, one has
  \begin{equation} \label{eq lambda}
   \langle \lambda_G(b) - \nu_G(b), \omega_{\Ghat,F}^\vee \rangle \in (-1,0].
  \end{equation}
 \end{lemdef}
 \begin{proof}
  Denote by $\Qhat \subset X^\ast(\That)$ the root lattice. Then the restriction $X^\ast(\That) \epi X^\ast(\Shat)$ canonically identifies the relative root lattice with $\Qhat_\Gamma$. Note that the preimage $w^{-1}(\kappa_G(b))_\Gamma$ in $X^\ast(\That)_\Gamma$ is a $\Qhat_\Gamma$-coset. Thus one has $\lambda' \succeq \lambda$ for two elements in $w^{-1}(\kappa_G(b))_\Gamma$ iff
  \[
   \langle \lambda', \omega_{\Ghat,F}^\vee \rangle - \langle \lambda, \omega_{\Ghat,F}^\vee \rangle \geq 0
  \]
  for all relative fundamental coweights $\omega_{\Ghat,F}^\vee$ of $\Ghat$ and moreover the left hand side always has integral value. Thus if a $\lambda_G(b)$ as in (\ref{eq lambda}) exists, it is the unique maximum. One easily constructs such a $\lambda_G(b)$ by choosing some $\lambda' \in w^{-1}(\kappa_G(b))_\Gamma$ and defining
  \[
   \lambda_G(b) \coloneqq \lambda' - \sum_{\hat\beta} \lceil \langle \lambda' - \nu_G(b),\omega_{\hat\beta}^\vee \rangle \rceil \cdot \hat\beta, 
  \]
  where the sum runs over all positive simple roots $\hat\beta \in \Qhat_\Gamma$ and $\omega_{\hat\beta}^\vee$ denotes the corresponding fundamental coweight. 
 \end{proof}

 \begin{example}
  Assume that $G = \GL_n$, $B$ is the upper triangular Borel subgroup and that $S=T$ is the diagonal torus. Then $\lambda_G(b)$ has the following geometric interpretation. 
  To an element $\nu \in \QQ^n \cong X_\ast(\Shat)_\QQ$, we associate a polygon $P(\nu)$ which is defined over $[0, n]$ with starting point $(0, 0)$ and slope $\nu_i$ over $(i - 1, i)$.
   We denote by $f_\nu$ the corresponding piecewise linear function. Then $P(\nu_G(b))$ is the (concave) Newton polygon of $b$ and $P(\lambda_G(b))$ is the largest polygon below $P(\nu_G(b))$ with integral slopes and break points.
    Indeed, the fundamental coweights of $\GL_n$ are given by $\omega_i = (\underbrace{1,\ldots,1}_{i\textnormal{ times}},\underbrace{0,\ldots,0}_{n-i \textnormal{ times}})$, thus
  \[
   \langle \lambda_G(b) - \nu_G(b), \omega_{i} \rangle = f_{\lambda_G(b)}(i) - f_{\nu_G(b)}(i),
  \]
  which implies $f_{\lambda_G(b)}(i) = \lfloor f_{\nu_G(b)}(i) \rfloor$ by (\ref{eq lambda}).

  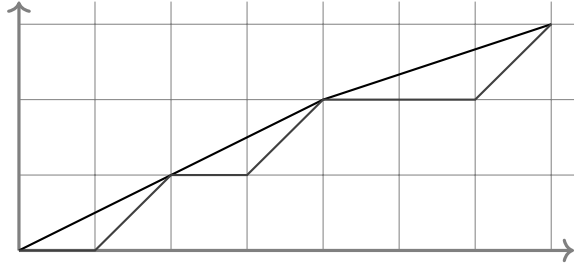
\begin{figure}[h] 
  \begin{center}
   \begin{tikzpicture}
    \draw[step=1cm,gray, ultra thin]  (0,0) grid (7.3,3.3);
    \draw[->,gray,very thick] (0,0) -- (7.3,0) ;
    \draw[->,gray,very thick] (0,0) -- (0,3.3) ;
        
    \draw[thick] (0,0) -- (4,2) -- (7,3) ;    
    
    \draw[thick,darkgray] (0,0) -- (1,0) -- (2,1) -- (3,1) -- (4,2) -- (5,2) -- (6,2) -- (7,3) ;
   \end{tikzpicture}
   \caption{The polygons associated to $\nu_G(b)$ and $\lambda_G(b)$ for $[b]$ given by $\nu_G(b)= (\frac{1}{2},\frac{1}{2},\frac{1}{2},\frac{1}{2},\frac{1}{3},\frac{1}{3},\frac{1}{3})$.}
  \end{center}
 \end{figure}
 \end{example}

 \begin{lemma} \label{lemma lambda}
  Let $f\colon  H \to G$ be a morphism of reductive groups over $\Ocal_F$. Then we have $\lambda_G(f(b)) = f(\lambda_H(b))$ in the following cases.
  \begin{subenv}
   \item $f$ is a central isogeny.
   \item $f$ is the embedding of a standard Levi subgroup, such that $\nu_H(b)$ is $G$-dominant. 
  \end{subenv}
 \end{lemma}
 \begin{proof}
  If $f$ is a central isogeny, we have $X^\ast(\That_H) = X^\ast(\That_G) \times_{\pi_1(G)} \pi_1(H)$ compatibly with the obvious Galois action and partial order on the right hand side. Thus $f$ and $\lambda$ commute.

  Now assume that $H$ is a standard Levi subgroup of $G$ and $\nu_H(b)$ is dominant, i.e.\ $\nu_H(b) = \nu_G(b)$. By \eqref{eq lambda} we have $-1 < \langle f(\lambda_H(b)) - \nu_G(b), \omega_{\Hhat,F}^\vee \rangle \leq 0$ for every relative fundamental coweight of $H$. Let $\omega_{\Ghat,F}^\vee$ be a relative fundamental coweight of $G$, but not of $H$. Then $\omega_{\Ghat,F}^\vee$ factorises through the center of $H$, thus for every quasi-character $\nu' \in X_\ast(\That)_\QQ$ the value of $\langle \nu',\omega_{\Ghat,F}^\vee \rangle$ is determined by the image of $\nu'$ in $\pi_1(H)_{\Gamma,\QQ}$. In $\pi_1(H)_{\Gamma,\QQ}$ we have equalities
  \[
   (\textnormal{image of } \nu_H(b)) = (\textnormal{image of } \kappa_H(b)) = (\textnormal{image of } \lambda_H(b)),
  \]
  thus $\langle \nu_H(b) - \lambda_H(b), \omega_{\Ghat,F}^\vee \rangle = 0$.
 \end{proof}

 \begin{notation}
 For fixed $b\in G(L)$ we denote by $\tilde\lambda \in X^\ast(\That)$ an arbitrary but fixed lift of $\lambda_G(b)$ and by $\lambda$ its image in $X^\ast(\Shat)$.
 \end{notation}

 \begin{remark} \label{rmk diag vs torus}
 Since $G$ is quasi-split, the maximal torus of the derived group $T^{\rm der}$ is induced and hence $\widehat{T^{\rm der}}^\Gamma \subseteq \Shat$. Thus any two elements in $X^\ast(\That^\Gamma)$ with the same image in  $X^\ast(\Shat)$ differ by a central cocharacter and thus have a different image in $\pi_1(G)_\Gamma$. In particular
 \[
  \{\tilde\mu \in X^\ast(\That) \mid \restr{\tilde\mu}{\That^\Gamma} = \lambda_G(b), w_G(\tilde\mu) = w_G(\mu)\} = \{\tilde\mu \in X^\ast(\That) \mid \restr{\tilde\mu}{\Shat} = \lambda, w_G(\tilde\mu) = w_G(\mu)\}.
 \]
 Since $V_\mu(\tilde\mu) = 0$ unless $\tilde\mu \leq \mu$, this implies
 $
  V_\mu(\lambda_G(b)) = V_\mu(\lambda).
 $
 \end{remark}
 
 \subsection{A group theoretic definition of $\lambda_G$ in the split case}

 We denote by $\Wtilde=\Wtilde_G \coloneqq  (\Norm_G(T))(L)/T(O_L)$ the extended affine Weyl group of $G$. Recall that we have canonical isomorphisms $\Wtilde_G \cong X_\ast(T) \rtimes W \cong W_a \rtimes \Omega_G$ where $W_a$ denotes the affine Weyl group of $G$ and $\Omega_G \subset \Wtilde_G$ the set of elements stabilising the base alcove, which we choose as the unique alcove in the dominant Weyl chamber whose closure contains $0$. In particular, we can lift the length function $\ell$ on $W_a$ to $\Wtilde_G$. 

 The embedding $\Norm_G(T) \mono G$ induces a natural map $B(\Wtilde_G) \to B(G)$, where $B(\Wtilde_G)$ denotes the set of $\Wtilde_G$-$\sigma$-conjugacy classes in $\Wtilde_G$. In general the notion of $\Wtilde_G$-conjugacy is finer than the notion of $G(L)$-conjugacy. Hence we consider only a certain subset of $B(\Wtilde_G)$.

 \begin{definition}
  \begin{subenv}
   \item We call $x \in \Wtilde_G$ basic if it is contained in $\Omega_G$. A $\sigma$-conjugacy class $O \in B(\Wtilde_G)$ is called basic if it contains a basic element.
   \item An element $x \in \Wtilde_G$ is called $\sigma$-straight if it satisfies
   \[
    \ell(x \sigma(x) \cdots \sigma^{n-1}(x)) = \ell(x) + \ell(\sigma(x)) + \cdots + \ell(\sigma^{n-1}(x)).
   \]
   for any non-negative integer $n$. Note that the right hand side might also be written as $n\cdot \ell(x)$. A $\sigma$-conjugacy class $O \in B(\Wtilde_G)$ is called straight if it contains a $\sigma$-straight element.
  \end{subenv}
 \end{definition}

 He and Nie gave a characterisation of the set of straight $\sigma$-conjugacy classes which is analogous of Kottwitz' description of $B(G)$ in \cite[\S~6]{kottwitz85}.

 \begin{proposition}[{\cite[Prop.~3.2]{HN14}}] \label{prop red to basic}
  A $\sigma$-conjugacy class $O \in B(\Wtilde_G)$ is straight if and only if it contains a basic $\sigma$-conjugacy class $O' \in B(\Wtilde_M)$ for some standard Levi subgroup $M \subset G$. 
 \end{proposition}

 Finally, by {\cite[Thm.~3.3]{HN14}} each $[b] \in B(G)$ contains a unique straight $O_{[b]} \in B(\Wtilde_G)$.

 We obtain the following description of $\lambda_G$ in the split case.

 \begin{proposition} \label{prop definition of lambda}
  Let $G$ be a split group over $\Ocal_F$, let $b \in G(L)$ and let $x \in O_{[b]}$ be a $\sigma$-straight element. Denote by $\lambda'$ its image under the canonical projection $\Wtilde_G \to X_\ast(T)$. Then $\lambda'_\dom = \lambda_G(b)_\dom$.
 \end{proposition}
 \begin{proof}
  By Proposition~\ref{prop red to basic} there exists a standard Levi subgroup $M \subset G$ and an $M$-basic element $x_M \in \Omega_M$ such that $x$ and $x_M$ are $\Wtilde_G$-conjugate. By \cite[Prop.~4.5]{HN15} any two such elements are even $W$-conjugate and thus correspond to the same element in $X_\ast(T)_{\dom}$. Since the same holds true for $\lambda_G(b)_\dom$ by Lemma~\ref{lemma lambda}, it suffices to prove the proposition in the basic case, i.e.\ when $\nu_G(b)$ is central. 
  
  If $[b]$ is basic, then $x$ is basic, thus $\lambda'$ is the (unique) dominant minuscule character with $w(\lambda') = \kappa_G(b)$ (cf.~\cite[Ch.~VI \S~2 Prop.~6]{bourbaki68}). Hence it suffices to show that $\lambda_G(b)$ is minuscule. By Lemma~\ref{lemma lambda}~(2) we may assume that $G$ is of adjoint type. This leaves finitely many cases, which can easily be checked using the explicit description of root systems in \cite{bourbaki68}.
 \end{proof}

 \section{Equidimensionality}\label{secequidim}

 While it is conjectured that $X_\mu(b)$ is equidimensional (cf. \cite[Conj.~5.10]{rapoport05}), this has not yet been proven in all cases. We give a partial result after reviewing the necessary geometry of $X_\mu(b)$ first.

 \subsection{Connected components} \label{sect connected components}
 
 Let $w_G\colon  G(L) \to \pi_1(G)$ be the Kottwitz homomorphism, as considered in \cite{kottwitz85}, compare Section \ref{sec21}. It induces a map $\Grass_G(k) \to \pi_1(G)$. After base change to $\Spec k$, this induces isomorphisms $\pi_0(LG_{k}) \cong \pi_0(\Grass_{G,k}) \cong \pi_1(G)$, compare Pappas and Rapoport \cite[Thm.~0.1]{PR08} in the equal characteristic case and Zhu \cite[Prop.~1.21]{zhu} in the mixed characteristic case. Here we used that as $G$ is unramified, the action of the inertia subgroup of the absolute Galois group of $F$ on $\pi_1(G)$ is trivial.
 
For $\omega \in \pi_1(G)$, we let $LG^\omega$ and $\Grass_G^\omega$ be the corresponding connected components. Denote for any subgroup $H \subset LG_k$ and subscheme $X \subset \Grass_{G,k}$ the intersection $H^\omega \coloneqq H \cap LG^\omega$ and $X^\omega \coloneqq X \cap \Grass_G^\omega$.

 In particular, $X_\mu(b)^\omega$ is a union of connected components, and the $J_b(F)$-orbit of $X_\mu(b)^\omega$ equals $X_\mu(b)$ by \cite[Thm.~1.2]{Nie} (see also \cite[Thm.~1.2]{CKV15}) whenever $X_{\mu}(b)^\omega$ is non-empty. One can even show that under some mild condition on the triple $(G,[b],\mu)$ every connected component of $X_\mu(b)$ is of the form $X_\mu(b)^\omega$ (\cite[Thm.~1.1]{Nie}, see also \cite[Thm.~1.1]{CKV15}), but we will not need this result. 
 
 The following general result on affine flag varieties is formulated in greater generality than needed in this paper. We will only apply it in the case where $H = G$ is a reductive group scheme. For consistency we denote affine flag varieties by the same symbol $\Grass$ as affine Grassmannians.
 
 \begin{proposition} \label{prop isom components of affine flag varieties}
  Let $f\colon H' \to H$ be a morphism of parahoric group schemes over $\Ocal_F$ such that the induced homomorphism on their adjoint groups is an isomorphism. Then the induced morphism on connected components of affine flag varieties
  \[
   f_{\Grass}^\omega \colon \Grass_{H'}^{\omega} \to \Grass_H^{f(\omega)}
  \]
  is a universal homeomorphism.
 \end{proposition} 
 
 \begin{proof}
  This is proven in \cite[\S~6]{PR08} if $\cha F = p$ and $p$ does not divide the order of $\pi_1(H'_{\rm der})$ or $\pi_1(H_{\rm der})$ (see also \cite[Prop.~4.3]{HZ} for the statement if $\cha F = 0$). We briefly recall the proof in \cite{PR08} and explain how to generalise it.
  
  Note that it suffices to show that $f_{\Grass}^\omega$ is bijective on geometric points. Indeed, it is a morphism of ind-proper ind-schemes (\cite[Cor.~2.3]{Ri16} if $\cha F = p$, \cite[\S~1.5.2]{zhu} if $\cha F = 0$) and thus universally closed.
  
  By homogeneity under the action of $H'(L)$ (resp. $H(L)$) we may assume $\omega = 0$. Denote by $H_{\rm der}$ the derived group of $H$ and by $\tilde{H}$ the simply connected cover of $H_{\rm der}$. Since we have a commutative diagram
  \begin{center}
   \begin{tikzcd}
   H'\arrow[bend left]{rrrr}{f} & H'_{\rm der} \arrow[hook]{l} & \Htilde' = \Htilde \arrow[two heads]{l} \arrow[two heads]{r} & H_{\rm der} \arrow[hook]{r}  & H
   \end{tikzcd}
  \end{center}
  it suffices to prove the theorem in the following two special cases.
  
  \emph{Case 1: $H' = H_{\rm der}$}. One can show that $f_{\Grass}^0$ is universally bijective using the argument in \cite[p.~144]{PR08}.
  
  \emph{Case 2: $H$ is semisimple and $H' = \Htilde$}. The following argument can be found in \cite[p.~140-141]{PR08}. Fix an algebraically closed field $l \supset k$ and let $M\supset L$ be the corresponding field extension of ramification index $1$. We denote by $Z$ the kernel of $\Htilde \to H$ and let $T$ and $\Ttilde$ denote the N\'eron models of fixed maximal tori in $H_F$ and $\Htilde_F$ satisfying $\Ttilde_F = f^{-1}(T_F)$. Since $\Htilde_F$ is simply connected, $\Ttilde_F$ is an induced torus, i.e.\ there exist finite field extensions $F_i/F$ such that
   \[
    \Ttilde^0 \cong \prod_i \Res_{\Ocal_{F_i}/\Ocal_F} \GG_m,
   \]
  thus there exists an $n \in \NN$ such that
   \[
    Z_F \subset \prod_i \Res_{F_i/F} \mu_n.
   \]
   In particular, we have $Z(M) \subset \Ttilde^0(\Ocal_M)$. Since $\Ttilde^0 \subset \Htilde$, $f_{\Grass}^0$ is injective on geometric points. The surjectivity is a direct consequence of \cite[Appendix, Lemma~14]{PR08}.
 \end{proof} 
 
 \begin{remark}
  If $\cha F = p$ and $p$ does not divide the order of $\pi_1(H'_{\rm der})$ or $\pi_1(H_{\rm der})$, it is shown in \cite[\S~6]{PR08} that $f_{\Grass}^\omega$ even induces an isomorphism of the underlying reduced ind-schemes. However, they show in \cite[Ex.~6.4]{PR08} that this is not necessarily the case when we drop the condition on $p$. On the other hand $f_{\Grass}^\omega$ is always an isomorphism in the case $\cha F = 0$, since universal homeomorphisms of perfect schemes are isomorphisms by \cite[Lemma~3.8]{BS}.
 \end{remark}

 Let $G^\ad$ be the adjoint group of $G$. We denote by a subscript ``ad'' the image of an element of $G(L)$, $X_\ast(T)$ or $\pi_1(G)$ in $G^\ad(L), X_\ast(T^\ad)$ or $\pi_1(G^\ad)$, respectively. By \cite[Cor. 2.4.2]{CKV15}, the homeomorphism of Proposition~\ref{prop isom components of affine flag varieties} induces a universal homeomorphism
 \begin{equation} \label{eq isom adj}
  X_\mu(b)^\omega \to X_{\mu_\ad}(b_\ad)^{\omega_\ad}
 \end{equation}
 whenever $X_\mu(b)^\omega$ is non-empty.

 \subsection{Equidimensionality for some affine Deligne-Lusztig varieties}\label{subsecequi}
 Equidimensionality is known to hold in the following cases.

 \begin{theorem}
 Let $G,b,\mu$ be as above.
  \begin{subenv}
   \item If $\cha F = p$, then $X_{\mu}(b)$ and $X_{\preceq \mu}(b)$ are equidimensional.  Furthermore, $X_{\preceq \mu}(b)$ is the closure of $X_{\mu}(b)$.
   \item Let $F$ be an unramified extension of $\QQ_p$, and let $G$ be classical, $\mu$ minuscule, and either $p \not=2$ or all simple factors of $G^\ad$ of type $A$ or $C$. Then $X_\mu(b)$ is equidimensional.
  \end{subenv}
 \end{theorem}
 \begin{proof}
  Assume first that $\cha F = p$. In the case where $G$ is split the assertion is proven in \cite[Cor.~6.8]{HV12} by identifying the formal neighbourhood of a closed point in the affine Deligne-Lusztig variety with a certain closed subscheme in the deformation space of a local $G$-shtuka. We briefly explain how to generalise the arguments in the proof of \cite[Cor.~6.8]{HV12} to arbitrary reductive group schemes over $\Ocal_F$. 
  
  The main ingredient is the following result in \cite{VW}, generalising \cite[Thm.~6.6]{HV12}. Let $x = gK \in X_{\preceq \mu}(b)(k)$ and denote $b' \coloneqq gb\sigma(g)^{-1}$. Consider the deformation functor
  \begin{align*}
   {\Dscr}ef_{b',0}\colon  (Art/k)\to & (Sets),\\
   A\mapsto & \{\btilde\in (\overline{K\mu(\epsilon)K})(A) \textnormal{ with }\btilde_{k}=b'\}/_{\cong}
  \end{align*}
  where $\btilde \cong \btilde'$ if there is an $h\in G(A\pot{\epsilon})$ with $h_k=1$ and $h^{-1}\btilde \sigma(h)= \btilde'$. By \cite[Prop.~2.6]{VW} this functor is pro-represented by the formal completion of $K\backslash K \mu(\epsilon) K$ at $b'$. Moreover, the universal object has a unique algebraisation by \cite[Lemma~2.8]{VW}. We denote by $D_{b',0}$ the algebraisation of $(K\backslash K \mu(\epsilon) K)^\wedge_{b'}$ and by $\btilde \in LG(D_{b',0})$ a lift of the universal object. We denote by $N_{[b],0} \subset D_{b',0}$ the minimal Newton stratum, that is the set of all geometric points $\sbar\colon  \Spec k_{\sbar} \to D_{b',0}$ such that $\btilde_{\sbar}$ is $G(k_{\sbar}\rpot{\epsilon})$-$\sigma$-conjugate to $b$ (or $b'$). Since $N_{[b],0}$ is closed, we may equip it with the structure of a reduced subscheme. By \cite[Thm.~2.9,~2.11]{VW} there exists a surjective finite morphism
  \[
   \Spec k\pot{x_1,\ldots,x_{2\langle\rho_G,\nu_G(b)\rangle}} \hat\times X_{\preceq\mu}(b)^{\wedge}_x \to N_{[b],0}
  \]
  where $\rho_G$ denotes the half-sum of all absolute positive roots in $G$ and $X_{\preceq\mu}(b)^{\wedge}_x$ the algebraisation of the completion of $X_{\preceq\mu}(b)$ in $x$. In particular, we get
  \begin{align*}
   \dim N_{[b],0} &= 2\langle \rho_G, \nu_G(b)\rangle + \dim  X_{\preceq\mu}(b)^{\wedge}_x \\
                  &\leq \langle \rho_G, \mu+\nu_G(b) \rangle - \frac{1}{2} \defect_G(b).
  \end{align*}
  Here the last inequality follows from the dimension formula of $X_{\preceq \mu}(b)$ in \cite[Thm.~1.1]{hamacher15a} and equality holds if and only if $\dim X_{\preceq\mu}(b)^{\wedge}_x = \dim X_{\preceq\mu}(b)$. The Newton stratification on $D_{b',0}$ satisfies strong purity in the sense of \cite[Def.~5.8]{viehmann}. Indeed, this is shown for $G=\GL_n$ in \cite[Thm.~7]{viehmann13} and the general case follows by \cite[Prop.~2.2]{hamacher}. Thus the conditions of \cite[Lemma~5.12]{viehmann} are satisfied and we get the dimension formula and closure relations of all Newton strata in $D_{b',0}$. In particular,
  \[
   \dim N_{[b],0} = \langle \rho_G, \mu+\nu_G(b) \rangle - \frac{1}{2} \defect_G(b).
  \]
  Thus $\dim X_{\preceq\mu}(b)^\wedge_x = \dim X_{\preceq\mu}(b)$ and since $x$ was an arbitrary closed geometric point of $X_{\preceq\mu}(b)$, this implies equidimensionality. Since $\dim X_{\preceq\mu'}(b) < \dim X_{\preceq\mu}(b)$ for every $\mu' \prec \mu$ by \cite[Thm.~1.1]{hamacher15a} this also implies the equidimensionality of $X_{\preceq\mu}(b)$ and that $X_\mu(b)$ is dense in $X_{\preceq\mu}(b)$.
  
  Now consider $F=\QQ_p$, $p \not= 2$ and assume first that there exists a faithful representation $\rho\colon  G \mono \GL_n$ such that the action of $\GG_m$ via $\rho(\mu)$ has weights $0$ and $1$. Then we can associate a Rapoport-Zink space of Hodge-type $\Mscr_{G,\mu}(b)$ to the triple $(G,\mu,b)$, whose perfection equals $X_\mu(b)$ by \cite[Thm.~3.10]{zhu}. Since $\Mscr_{G,\mu}(b)$ is equidimensional by \cite[Thm.~1.3]{hamacher}, so is $X_\mu(b)$. 

  Now the morphism $X_\mu(b) \to X_{\mu_\ad}(b_\ad)$ induced by the the canonical projection $G \epi G^\ad$ is an isomorphism on connected components  by (\ref{eq isom adj}). Thus all connected components of $X_{\mu_\ad}(b_\ad)$ which are contained in the image of $X_\mu(b)$ are equidimensional. Since all connected components are isomorphic to each other by \cite[Thm.~1.2]{CKV15}, this implies that $X_{\mu_\ad}(b_\ad)$ is equidimensional. Thus any affine Deligne-Lusztig variety with $G$ classical, adjoint and $\mu$ minuscule is equidimensional. Applying (\ref{eq isom adj}) once more, the claim follows for $p \not=2$.  If $p=2$, the spaces $\Mscr_{G,\mu}(b)$ are only defined if $(G,\mu,b)$ is of PEL-type, but in this case the rest of the proof is identical.

  If $F$ is an unramified field extension of $\QQ_p$, let $G' = \Res_{O_F/\ZZ_p} G$ and $\mu' = (\mu,0, \ldots, 0), b' = (b,1,\ldots ,1)$ with respect to the identification $G'_L \cong \prod_{F \mono L} G$. By \cite[Lemma~3.6]{zhu} and the Cartesian diagram below it, we have $X_{\mu'}(b') \cong X_\mu(b)$. Thus $X_\mu(b)$ is equidimensional.
 \end{proof}

\section{Irreducible components in the superbasic case}\label{secsb}
In this section we prove Theorem \ref{thm_main} for superbasic $\sigma$-conjugacy classes. In \cite[\S~8]{hamacher15a} this has been reduced to a purely combinatorial statement, which we prove using the bijectivity of sweep maps on rational Dyck paths.

\subsection{Superbasic $\sigma$-conjugacy classes}

An element $b\in G(L)$ or the corresponding $\sigma$-conjugacy class $[b]\in B(G)$ is called superbasic if no element of $[b]$ is contained in a proper Levi subgroup of $G$ defined over $F$.

\begin{remark}[{\cite[\S~3.1]{CKV15}}]\label{remsb}

\begin{enumerate}
\item If $b$ is superbasic in $G(L)$ then the simple factors of the adjoint group $G^{\ad}$ are of the form $\Res_{F_{d}\mid F}\PGL_{n}$ for unramified extensions $F_{d}$ of $F$ (of degree $d$) and $n\geq 2$. In particular, $X_\mu(b)$ is equidimensional if $\cha F = p$ or $F$ is an unramified extension of $\QQ_p$.
\item For every $[b]\in B(G)$ there is a standard parabolic subgroup $P\subset G$ defined over $F$ and with the following property. Let $T$ be a fixed maximal torus of $G$, and $M$ the Levi factor of $P$ containing $T$. Then there is a $b\in [b]\cap M(L)$ which is superbasic in $M$.
\end{enumerate}
\end{remark}

We first consider the special case where $[b]$ is superbasic and where $G$ is of the form $\Res_{F_{d}\mid F}\GL_{n}$ for some $d,n$. In this case we give a proof using EL-charts as in \cite{hamacher15a} (see also \cite{dJO00} for the split case). We then reduce the general superbasic case to this particular case. 

For $G$ as above $L \otimes_F F_d \cong \prod_{\tau\colon F_d \mono L} L$ yields an identification
 \[
  G(L) = \prod_{\tau \in I} \GL_n(L)
 \]
mapping $g\in G(L)$ to a tuple $(g_\tau)_{\tau\in I}$ where $I \coloneqq \Gal(F_d,F) \cong \ZZ/d\ZZ$. Let $S \subset T \subset B \subset G$ be the split diagonal torus, the diagonal torus and the upper triangular Borel, respectively. We have a canonical identification $X_\ast(T) \cong (\ZZ^n)^{|I|}$. Then the dominant elements in $X_\ast (T)$ are precisely the $\mu = (\mu_\tau)_{\tau \in I} \in X_\ast(T)$ such that the components of $\mu_\tau$ are weakly decreasing for each $\tau$.
 
 We identify $X_\ast(S)$ with the invariants $X_\ast(T)^I = \ZZ^n$, thus
 \[
  \restr{\mu}{\Shat} = \sum_{\tau \in I} \mu_\tau.
 \]
 Moreover, this identifies the partial order $\leq$ on $X_\ast(S)_\QQ$ with the dominance order on $\QQ^n$.

\subsection{A combinatorial identity} \label{sect sweep}

 An important tool when considering the combinatorics of EL-charts is the sweep map defined by Armstrong, Loehr and Warrington in \cite{ALW15}. We need a multiple component version of it, which turns out to be easily realised as a special case of the classical sweep map.
 
 \begin{notation}
  By a word $\wbf$ we mean a finite sequence of integers $w_1 \cdots w_r$. For $1\leq k\leq r$ we define the level of $\wbf$ at $k$ by $l(\wbf)_k \coloneqq \sum_{i=1}^k w_i$. We consider the following sets for fixed sequences of integers $a_{\tau,1},\ldots,a_{\tau,n}$ where $ 1 \leq \tau \leq d$.
  \begin{subenv}
   \item Let $\Acal_\ZZ^{(d)}$ denote the set of words $\wbf = w_1 \dotsm w_{d\cdot n}$ such that the sub-word $\mathbf{w}_{(\tau)} \coloneqq w_\tau w_{\tau+d} \cdots w_{\tau + (n-1) \cdot d}$ is a rearrangement of $a_{\tau,1},\ldots,a_{\tau,n}$ for any $\tau\in \{1,\dotsc, d\}$.
   \item Denote by $\Acal_\NN^{(d)} \subset \Acal_\ZZ^{(d)}$ the subset of words whose level at multiples of $d$ is non-negative. Following \cite{williams} and \cite{ALW15}, we call its elements ($d$-component) Dyck words.
  \end{subenv}
 \end{notation}

 \begin{definition}
  The sweep map $sw^{(d)}\colon  \Acal_\ZZ^{(d)} \to \Acal_\ZZ^{(d)}$ is the map that sorts $\wbf$ according to its level by permuting $\wbf_{(\tau)}$ using the following algorithm. Initialise $sw^{(d)}(\wbf)_{(\tau)} = \emptyset$ for any $1 \leq \tau \leq d$. For each $a$ down from $-1$ to $-\infty$ and then down from $\infty$ to $0$ read $\wbf_{(\tau)}$ from right to left and append to $sw^{(d)}(\wbf)_{(\tau)}$ all letters $w_k$ such that $l(\wbf)_k = a$.
  \end{definition}

 We deduce the bijectivity of $sw^{(d)}$ from Williams' result for the classical sweep map in \cite{williams}.
  
 \begin{proposition} \label{prop sweep}
  $sw^{(d)}$ is bijective and preserves $\Acal_\NN^{(d)}$.
 \end{proposition}
 \begin{proof}
  If $d=1$, the map $sw^{(1)}$ is precisely the sweep map defined in \cite{williams} and the proposition is proven in \cite[Thm.~6.1,~6.3]{williams}. In order to reduce to this case, we need to construct an injection $\Acal_\ZZ^{(d)} \mono \Acal_\ZZ^{(1)}$ which identifies Dyck words and preserves the sweep map, i.e.\ the diagram
  \begin{equation}
   \begin{tikzcd} \label{diag sweep}
    \Acal_\ZZ^{(d)} \arrow[hook]{r} \arrow{d}{sw^{(d)}} & \Acal_\ZZ^{(1)} \arrow{d}{sw^{(1)}} \\
    \Acal_\ZZ^{(d)} \arrow[hook]{r} & \Acal_\ZZ^{(1)}
   \end{tikzcd}
  \end{equation}
  commutes. Note that part of this construction is also the choice of a sequence $\{a'_1,\ldots, a'_{n\cdot d}\}$ for $\Acal_\ZZ^{(1)}$.
  
  As preparation, fix an integer $N$ big enough such that for any $\mathbf{w} \in \Acal_\ZZ^{(d)}$ and $1 \leq \tau \leq d$ as above the following inequalities hold.
  \begin{align}\label{eq sweep 1}
   \min \{l(\wbf)_{k} +N \mid k \equiv \tau \pmod d\} &> \max \{l(\wbf)_{k} \mid k \equiv \tau-1 \pmod d\} \\
   \min \{l(\wbf)_{k} + \tau\cdot N \mid k \equiv \tau \pmod d\}  &\geq 0. \label{eq sweep 2}
  \end{align}
  We now construct a map $\Acal_\ZZ^{(d)} \to \Acal_\ZZ^{(1)}, \wbf \mapsto \wbf^{+N}$ satisfying the conditions above as follows. For given $\mathbf{w}$, let $\wbf^{+N}$ be the word which one obtains by replacing $w_{\tau+(i-1)\cdot d}$ by
  \begin{align*}
   w'_{\tau+(i-1)\cdot d} &\coloneqq \begin{cases}
    w_{\tau+(i-1)\cdot d}+ N & \textnormal{if } \tau \neq d \\
    w_{\tau+(i-1)\cdot d} - N \cdot (d-1) & \textnormal{if } \tau =d
   \end{cases} 
  \shortintertext{for $1\leq i \leq n$ and $1 \leq \tau \leq d$. 
  Then $\wbf^{+N} \in \Acal_\ZZ^{(1)}$ for the choice $\{a'_1,\ldots, a'_{n\cdot d}\}$, where}  
   a'_{\tau+(i-1)\cdot d} &\coloneqq    \begin{cases}
                                 a_{\tau,i} +N & \textnormal{if } \tau \neq d \\
                                 a_{\tau,i}-N\cdot (d-1) &\textnormal{if } \tau = d.
                                \end{cases}
  \end{align*}

   The map $\wbf \mapsto\wbf^{+N}$ is obviously injective. Note that for any $k$ we have $l(\wbf^{+N})_k = l(\wbf)_k + \bar{k} \cdot N$ where $0 \leq \bar{k} \leq d-1$ denotes the residue of $k$ modulo $d$. Thus $l(\wbf^{+N})_k = l(\wbf)_k$ if $k$ is a multiple of $d$, and $l(\wbf^{+n}) \geq 0$ by (\ref{eq sweep 2}) otherwise. Hence $\wbf^{+N} \in \Acal_\NN^{(1)}$  if any only if $\wbf \in \Acal_\NN^{(d)}$. 
   
    By (\ref{eq sweep 1}), we have $\min_i l(\wbf^{+N})_{\tau +(i-1)\cdot d}  > \max_i l(\wbf^{+N})_{\varsigma+ (i-1)\cdot d}$ for all $1 \leq \tau < \varsigma \leq d-1$ or $\varsigma < \tau = d$. Thus the permutation of letters of $\wbf^{+N}$ induced by the classical sweep map decomposes into a product of permutations of the subsets $\{\tau,\tau+d,\tau+2d,\ldots \}$. Since moreover $l(\wbf^{+N})_{\tau + (i-1) \cdot d} \geq l(\wbf^{+N})_{\tau + (j-1) \cdot d}$ iff $l(\wbf)_{\tau + (i-1) \cdot d} \geq l(\wbf)_{\tau + (j-1) \cdot d}$, the permutations induced by the classical sweep map $sw^{(1)}$ applied to $\wbf^{+N}$ and $sw^{(d)}$ applied to $\wbf$ coincide. In other words, the diagram (\ref{diag sweep}) commutes.
 \end{proof}

\subsection{Characterisation of EL-charts}

 Throughout the section, we fix a positive integer $m$ coprime to $n$ and denote $\restr{\nu}{\Shat} = (\frac{m}{n},\ldots,\frac{m}{n}) \in X^\ast(\Shat) = X_\ast(T)^I$. Let $m_1,\ldots,m_d$ be arbitrary integers such that $m_1 + \dotsm +m_d=m$. We shall later make convenient choices of them depending on $\mu$. We recall the notion of EL-charts as they were presented in \cite[\S~5]{hamacher15a}.

Let $\ZZ^{(d)} \coloneqq \coprod_{\tau \in I} \ZZ_{(\tau)}$ be the disjoint union of $d$ copies of $\ZZ$.  We impose the notation that for any subset $A \subset \ZZ^{(d)}$ we write $A_{(\tau)} \coloneqq A\cap  \ZZ_{(\tau)}$. For $x\in\ZZ$ we denote by $x_{(\tau)}$ the corresponding element of $\ZZ_{(\tau)}$ and write $|a_{(\tau)}| \coloneqq a$.
 We equip $\ZZ^{(d)}$ with a partial order ``$\leq$'' defined by
 \[
  x_{(\tau)} \leq y_{(\varsigma)} \colon \Leftrightarrow \tau = \varsigma \textnormal{ and } x\leq y
 \]
 and a $\ZZ$-action given by
 \[
  x_{(\tau)} + z \coloneqq (x+z)_{(\tau)} .
 \]
 Furthermore we consider a $\ZZ$-equivariant function $f\colon \ZZ^{(d)} \to \ZZ^{(d)}$ with 
 \[
  f(a_{(\tau)}) = a_{(\tau+1)} + m_\tau
 \]
In particular, $f(\ZZ_{(\tau)}) = \ZZ_{(\tau+1)}$ and $f^d(a) = a+m$.  
 \begin{definition}
  \begin{subenv}
   \item An EL-chart is a non-empty subset $A \subset \ZZ^{(d)}$ which bounded from below and satisfies $f(A)\subset A$ and $A+n\subset A$.
   \item Two EL-charts $A,A'$ are called equivalent, if there exists an integer $z$ such that $A+z = A'$. We write $A \sim A'$.
  \end{subenv}
 \end{definition}

 Let $A$ be an EL-chart and $B = A\setminus (A+n)$. It is easy to see that $\#B_{(\tau)} = n$ for all $\tau \in I$. We define a sequence $b_0,\ldots,b_{d\cdot n}$ as follows. Let $b_0 = b_{n\cdot d} = \min B_{(0)}$ and for given $b_i$ let $b_{i+1} \in B$ be the unique element of the form
 \[
  b_{i+1} = f(b_i) - \mu'_{i+1} \cdot n
 \]
 for a non-negative integer $\mu'_i$. These elements are indeed distinct: If $b_i = b_j$ then obviously $i \equiv j \pmod d$ and then $b_{i + k\cdot d} \equiv b_i + k\cdot m \pmod n$ implies that $i=j$ as $m$ and $n$ are coprime.
 
 It will later be helpful to distinguish the $b_i$s and $\mu'_i$ of different components. For this we change the index set to $I \times \{1,\ldots,n\}$ via
 \begin{eqnarray*}
  b_{\tau,i} &\coloneqq& b_{\tau + (i-1)\cdot d} \\
  \mu'_{\tau,i} &\coloneqq&  \mu'_{\tau + (i-1)\cdot d}.
 \end{eqnarray*} 
 Here we choose the set of representatives $\{1,\ldots,d\} \subset \ZZ$ of $I$.

 \begin{definition}
   With the notation above, $\mu'$ is called the type of $A$.
 \end{definition}
 
 \begin{remark}
   This definition differs slightly from the definition of the type in \cite[p.~12822]{hamacher15a}. In this article we choose the indices such that
   $\mu'_{\tau,i}$ measures the difference between between $b_{\tau,i}$ and $b_{\tau-1,i}$ while in \cite{hamacher15a} it yields the difference between $b_{\tau,i}$ and $b_{\tau+1,i}$. Since one can alternate between those two notions by replacing $f$ by $\fbar \coloneqq f^{-1}$ and $\mu$ by $(-\mu)_{\rm dom}$, we can still use the combinatorial results of \cite{hamacher15a}. Moreover, we consider the Borel of \emph{upper} triangular matrices instead of lower triangular matrices in loc.~cit., thus inverting the order on $X_\ast(S)$ and $X_\ast(T)$.
 \end{remark}
 
 The type characterises an EL-chart up to equivalence.
 
 \begin{lemma}[{\cite[Lemma 5.3]{hamacher15a}}] \label{lem type}
  Let
  \[
   P_{m,n,d} \coloneqq \{ \mu' \in \left(\ZZ_{\geq 0}^n\right)^{|I|} \mid \restr{\mu'}{\Shat} \leq \restr{\nu}{\Shat} \}.
  \]
  Then the type of any EL-chart $A$ lies in $P_{m,n,d}$ and the type defines a bijection
  \[
   \bigslant{\{\textnormal{ EL-charts }\} }{\sim} \leftrightarrow P_{m,n,d}.
  \] 
 \end{lemma} 

 \begin{example} \label{ex EL-charts}
  There are two important special cases of EL-charts.
  \begin{assertionlist}
   \item An EL-chart is called small if $A+n \subset f(A)$, in other words if its type only has entries $0$ and $1$. They correspond to the affine Deligne-Lusztig varieties with minuscule Hodge point.
   \item A semi-module is an EL-chart $A \subset \ZZ$. These are the invariants that occur in the split case.
  \end{assertionlist}
  There is a bijection between small semi-modules up to equivalence and rational Dyck paths from $(0,0)$ to $(n-m,m)$, that is lattice paths allowing only steps in the north and east direction which stay above the diagonal. This gives a purely combinatorial motivation for the definitions below.
  
  The bijection is given as follows (see \cite{GM13} for more details). With a given equivalence class $[A]$ of small semi-modules, we associate the path which goes east at the $i$-th step if $\type(A)_i = 0$ and north if $\type(A)_i = 1$. By the above lemma, this map is well-defined and a bijection. Moreover, if we choose $\min A =0$, then one can recover $A$ from the Dyck path as the set of $(m,m-n)$-levels in the sense of \cite{ALW15} of points on or above the path, giving the inverse to the bijection.
 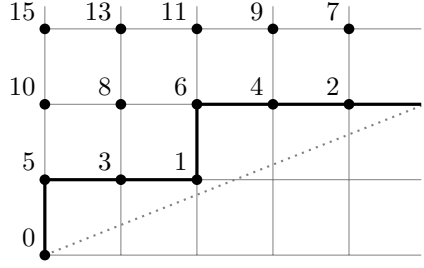
\begin{figure}[h] \label{fig N_0}
  \begin{center}
   \begin{tikzpicture}
    \draw[step=1cm,gray, ultra thin]  (0,0) grid (4.95,3.3);
        
    \draw[thick, gray, dotted] (0,0) -- (4.95,1.98) ;    
    
    \draw[very thick] (0,0) -- (0,1) ;
    \draw[very thick] (0,1) -- (2,1) ;
    \draw[very thick] (2,1) -- (2,2) ;
    \draw[very thick] (2,2) -- (4.95,2) ;

    \fill (0,0) circle (2pt) node[anchor = south east] {$0$} ;
    \fill (0,1) circle (2pt) node[anchor = south east] {$5$} ;
    \fill (0,2) circle (2pt) node[anchor = south east] {$10$} ;
    \fill (0,3) circle (2pt) node[anchor = south east] {$15$} ;
    
    \fill (1,1) circle (2pt) node[anchor = south east] {$3$} ;
    \fill (1,2) circle (2pt) node[anchor = south east] {$8$} ;
    \fill (1,3) circle (2pt) node[anchor = south east] {$13$} ;
    
    \fill (2,1) circle (2pt) node[anchor = south east] {$1$} ;
    \fill (2,2) circle (2pt) node[anchor = south east] {$6$} ;
    \fill (2,3) circle (2pt) node[anchor = south east] {$11$} ;
    
    \fill (3,2) circle (2pt) node[anchor = south east] {$4$} ;
    \fill (3,3) circle (2pt) node[anchor = south east] {$9$} ;
    
    \fill (4,2) circle (2pt) node[anchor = south east] {$2$} ;
    \fill (4,3) circle (2pt) node[anchor = south east] {$7$} ;    
    
   \end{tikzpicture}
   \caption{The associated Dyck path and $(5,-2)$-levels for $m=5$, $n=7$ and $A = \NN_0$}
  \end{center}
 \end{figure}
 \end{example}

 There is another invariant of EL-charts which is more important for the application of this theory, as it allows us to calculate the dimension of strata inside the affine Deligne-Lusztig variety.
 
 \begin{definition}
  Let $A$ be an EL-chart of type $\mu'$ and let $b_{\tau,i}$ be defined as above. For each $\tau \in I$ let $\tilde{b}_{\tau,1} >  \ldots > \tilde{b}_{\tau,n}$ be the elements of $B_{(\tau)}$ arranged in decreasing order. Define
  \[
   \tilde{\mu}_{\tau,i}  = \mu'_{\tau,i'}
  \]
  where $i$ is the unique number such that $\tilde{b}_{\tau,i} = b_{\tau,i'}$.
 We call $\tilde{\mu}$ the cotype of $A$.
 \end{definition}

 It is shown in \cite[p.~12831]{hamacher15a} that $\cotype(A) \in P_{m,n,d}$. Since the cotype is obviously invariant under equivalence, we obtain a map
 \[
  \zeta\colon  P_{m,n,d} \to P_{m,n,d}, \type(A) \mapsto \cotype(A).
 \]
 We claim that $\zeta$ is bijective. For this we note that $\zeta$ is the composition of
 \begin{equation} \label{eq zeta}
   \mu' \mapsto (w_k \coloneqq m_{k\ {\rm mod}\ d} - \mu'_k \cdot n)_{k=1,\ldots,n\cdot d} \stackrel{sw^{(d)}}{\mapsto} (\tilde{w}_k) \mapsto (\frac{m_{\tau} -\tilde{w}_{\tau+i\cdot d}}{n})_{\tau,i}.
 \end{equation}
 Thus its bijectivity follows from Proposition~\ref{prop sweep}.
 
 \begin{example}
  For $d=1, n=7, m=5$ and $A =\NN_0$, we can describe (\ref{eq zeta}) as follows. In Figure~1 one sees that $\mu' \coloneqq \type(A) = (0, 1, 1, 0, 1, 1, 1)$. This is mapped to the word $\wbf = (5, -2, -2, 5, -2, -2, -2)$, whose levels $l(\wbf) = (5, 3, 1, 6, 4, 2, 0)$ are the corresponding elements of $B \coloneqq A \setminus (A+n)$. Thus applying the sweep map, which sorts the letters of $\wbf$ according to their level, is nothing else than permuting the letters such that the corresponding elements of $B$ get arranged in decreasing order. Now $sw(\wbf) = (5, 5, -2, -2, -2, -2, -2)$, which yields $\zeta(\mu') = (0, 0, 1, 1, 1, 1, 1)$.
 \end{example}

 Altogether, we obtain the following theorem, which generalises the result of \cite[Cor.~6.4]{williams}. It was conjectured in \cite[Conj.~8.3]{hamacher15a} and in the split case by de Jong and Oort in \cite[Rem.~6.16]{dJO00}.
 \begin{theorem} \label{thm cotype}
  The cotype induces a bijection
  \[
   \bigslant{\{\textnormal{ EL-charts }\} }{\sim} \leftrightarrow P_{m,n,d}.
  \] 
 \end{theorem}

 \subsection{The superbasic case}
 Proposition~\ref{prop superbasic} is a direct consequence of Theorem~\ref{thm cotype}  together with the relation between orbits of irreducible components and EL-charts in \cite[\S~8]{hamacher15a}. We briefly recall this relation for the reader's convenience before proving Prop.~\ref{prop superbasic}.

 When applying the results of the previous subsection to affine Deligne-Lusztig varieties, we consider EL-charts satisfying certain additional criteria.
 
 \begin{definition}
  Let $A$ be an EL-chart.
  \begin{subenv}
   \item $A$ is called normalised if $\sum_{b\in B_{(0)}} |b| = {n \choose 2}$ where $B_{(0)} = A_{(0)} \setminus (A_{(0)} + n)$.
   \item The Hodge point of $A$ is defined as $\type (A)_\dom$.
  \end{subenv}
 \end{definition}
 
 Note that every EL-chart is equivalent to a unique normalised EL-chart. Let $P_{\mu} \coloneqq \{\mu' \in P_{m,n,d} \mid \mu'_\dom = \mu \}.$ Then by Lemma~\ref{lem type} $A \mapsto \type(A)$ induces a bijection
 \[
  \{\textnormal{normalized\ EL-charts with Hodge point } \mu\} \leftrightarrow P_{\mu}.
 \]
It is easy to see that $\zeta$ stabilises $P_{\mu}$. Thus Theorem~\ref{thm cotype} says that   $A \mapsto \cotype(A)$ induces a bijection between the set of normalised EL-charts with Hodge point  $\mu$ and $P_{\mu}$.

For every minuscule $\mu \in X_\ast(T)_\dom$ there exists a unique basic $\sigma$-conjugacy class in $B(G,\mu)$. We choose a representative of this $\sigma$-conjugacy class as follows. Let $m_\tau = \val\det \mu(\epsilon)_\tau$ and choose $b =((b_{\tau,i,j})_{i,j=1}^n)_{\tau \in I}$ with
 \[
  b_{\tau,i,j} = \begin{cases}
                  \epsilon^{\lfloor \frac{i+m_\tau}{n} \rfloor} & \textnormal{if } j-i \equiv m_\tau \pmod{n} \\
                  0 & \textnormal{otherwise.}
                 \end{cases}
 \]
 Then the invariants $\lambda,\nu \in X^\ast(\Shat) = X^\ast(\That)_\Gamma \cong \ZZ^n$ are given by  $\nu = (\frac{m}{d\cdot n}, \ldots, \frac{m}{d\cdot n})$ with $m = \sum_{\tau\in I} m_\tau$ and $\lambda = (\lambda_1, \ldots, \lambda_n)$ with $\lambda_i = \lfloor \frac{i\cdot m}{n} \rfloor - \lfloor \frac{(i-1) \cdot m}{n} \rfloor$. The requirement that $b$ is in fact superbasic corresponds to the assertion that $m$ and $n$ are coprime.

 By our choice of $b$, the variety $X_\mu(b)^0 \coloneqq X_\mu(b) \cap \Grass_G^0$ is non-empty. In \cite{hamacher15a},\cite{hamacher15b} we constructed a $J_b(F)^0$-invariant cellular decomposition
 \[
  X_\mu(b)^0 = \bigcup_A S_A
 \]
 where the union runs over all normalised EL-charts with Hodge-point $\mu$. 
 We denote
 \[
  V_A \coloneqq \{(i,j) \mid b_i < b_j; \mu'_i = \mu'_j +1 \}.
 \] 
 In \cite[Prop.~6.5]{hamacher15a}, \cite[Prop.~13.9]{hamacher15b} we show that $\AA^{V_A} \isom S_A$ by constructing an element $g_A \in LG(\AA^{V_A})$, respectively a basis $(v_{\tau,i})$ of the universal $G$-lattice over $S_A$ in the terminology of above articles, such that $S_A$ is its image in the affine Grassmannian. In particular $\dim S_A = \# V_A$.

 Following the calculations of the term $S_1$ in \cite[p.~12831]{hamacher15a}, one obtains $\#V_A$ from $\tilde{\mu}$ using the formula
 \[
  \# V_A = \sum \lfloor \langle \restr{\tilde\mu}{ \Shat} - \mu_{\rm adom}, \hat\omega_{F}^\vee \rangle \rfloor,
 \]
 where the sum runs over all relative fundamental coweights $\hat\omega_F^\vee$ of $\Ghat$ and $\mu_{\rm adom}$ denotes the anti-dominant element in the $W$-orbit of $\mu$. In particular, $S_A$ is top-dimensional if and only if $\restr{\cotype(A)}{\Shat} = \lambda$.

 \begin{proof}[Proof of Proposition~\ref{prop superbasic}]
  Let $G$ be arbitrary. We assume without loss of generality that $b \in K\mu(\epsilon)K$, thus $X_\mu(b)^0 \not= \emptyset$. Since $J_b(F)$ acts transitively on $\pi_0(X_\mu(b))$ by \cite[Thm.~1.2]{CKV15}, it suffices to construct $C_{\tilde\mu}^0 \coloneqq C_{\tilde\mu} \cap X_\mu(b)^0$, which have to be $J_b(F)^0$-stable and universally homeomorphic to  affine spaces of the correct dimension. In particular, we may take $C_{\cotype(A)}^0 = S_A$ if $G = \Res_{F_d/F} \GL_n$.

  By Remark~\ref{remsb} we have $G^\ad \cong \prod_{i=1}^n \Res_{F_{d_i}/F} \PGL_{n_i}$. Let $G' = \prod_{i=1}^n \Res_{F_{d_i}/F} \GL_{n_i}$ and $b',\mu'$ be lifts of $b_\ad$ and $\mu_\ad$ to $G'$, such that $X_{\mu'}(b')^0 \not= \emptyset$. We identify the underlying topological spaces $X_\mu(b)^0 = X_{\mu_\ad}(b_\ad)^0 = X_{\mu'}(b')^0$ via the homeomorphism (\ref{eq isom adj}). Thus we get a cellular decomposition of $X_\mu(b)^0$ per transport of structure from $X_{\mu'}(b')^0$. Since it is $J_{b'}(F)^0$-stable, we consider the canonical projections $J_b(F)^0 \xrightarrow{p} J_{b_\ad}(F)^0 \xleftarrow{q} J_{b'}(F)^0$. It suffices to show that $q$ is surjective (implying that the decomposition is $J_{b_\ad}(F)^0$-stable) and that the $J_b(F)^0$-action factors through $J_{b_\ad}(F)^0$.

  To prove the surjectivity, let $j \in J_{b_\ad}(F)^0$ and choose a preimage $g \in G(L)^0$ of $j$. The element $g$ satisfies $g^{-1}b\sigma(g)=zb$ for some $z\in Z'(L)\cap G(L)^0=Z(\mathcal O_L)$, where $Z'$ denotes the center of $G'$. We choose $z'\in Z'(\mathcal O_L)$ with $(z')^{-1}\sigma(z')=z^{-1}$. Then $gz' \in J_{b'}(F)^0$ maps to $j$, as claimed.

  Now  an elementary calculation of the kernel shows that we have an exact sequence
 \begin{center}
   \begin{tikzcd}
    1 \arrow{r} & Z(\Ocal_F) \arrow{r} & J_b(F)^0 \arrow{r} & J_{b_\ad}(F)^0,
   \end{tikzcd}
  \end{center}
  where $Z$ denotes the center of $G$. Since $Z(\Ocal_F)$ acts trivially on $\Grass_G$, the $J_b(F)^0$-action factors through $J_{b_\ad}(F)^0$, as claimed.
 \end{proof}

 \begin{corollary} \label{cor superbasic}
  Conjecture~\ref{conj main} is true if $b$ is superbasic and $\mu$ minuscule.
 \end{corollary}
 \begin{proof}
  We have
  \[
  J_b(F) \backslash \Sigma(X_\mu(b)) \cong \{ \tilde\mu \in P_{\mu} \mid C_{\tilde\mu} \textnormal{ top-dimensional} \} \cong \{ \tilde\mu \in W.\mu \mid \restr{\tilde\mu}{\Shat} = \lambda \}.
  \]
 \end{proof}

\section{Reduction to the superbasic case}

 In this section we consider the general case of Theorem \ref{thm_main}, i.e.~$G$ is an unramified reductive group over $F$, $\mu$ is minuscule, and $b$ is  an arbitrary element of $G(L)$. The goal is to use a reduction method, first introduced in \cite{GHKR06}, to relate to the superbasic case.

 Let $P \subset G$ be a smallest standard parabolic subgroup of $G,$ defined over $F$ and with the following property.  Let $M$ be the Levi factor of $P$ containing $T$. Then we want that $M(L)$ contains a $\sigma$-conjugate of $b$ which is superbasic in $M$. Fix a representative $b\in M(L)$ of $[b]_G=[b]$. Then we furthermore want that the $M$-dominant Newton point of $b$ is already $G$-dominant. For existence of such $P,M,b$ compare Remark \ref{remsb}. We write $P = M\cdot N$ where $N$ denotes the unipotent radical of $P$. Since $b\in M(L)$, this induces a decomposition
 \[
  J_b(F) \cap P(L) = (J_b(F) \cap M(L)) \cdot (J_b(F) \cap N(L)).
 \]
 
 Throughout the section, we may refer to subschemes of the loop group or Grassmannian by their $k$-valued points to improve readability, e.g.\ write $K$ instead of $L^+G$ or $N(L)$ instead of $LN$. We denote $K_M = M(\Ocal_L)$, $K_N = N(\Ocal_L)$ and $K_P = P(\Ocal_L)$.

 We consider the variety
$$ X_{\mu}^{M\subset G}(b) =\{gK_M \in \Grass_M \mid g^{-1}b\sigma(g)\in K\mu K\}.$$
 Then we have $  X_\mu^{M\subset G}(b) = \coprod_{\mu' \in I_{\mu,b}} X_{\mu'}^M(b)$
 where $I_{\mu,b}$ is the set of $M$-conjugacy classes of cocharacters $\mu'$ in the $G$-conjugacy class of $\mu$ with $[b]_M\in B(M,\mu')$. As $[b]_M$ is basic in $M$, this latter condition is equivalent to $\kappa_M(b)=\kappa_M(\mu')$ in $\pi_1(M)_{\Gamma}$. We identify an element of $I_{\mu,b}$ with its $M$-dominant representative in $X_*(T)$. Note that $I_{\mu,b}$ is non-empty and finite, but may have more than one element if $G$ is not split.

\begin{notation}
  Note that $X_\mu^{M\subset G}(b)$ is in general not equidimensional, although the individual summands are conjectured to be. We define
  $$\Sigma'(X_{\mu}^{M \subset G}(b)) \coloneqq
   \bigcup_{\mu'\in I_{\mu,b}}\Sigma^{\top}(X_{\mu'}^{M}(b)).
  $$
\end{notation}

Using Corollary~\ref{cor superbasic} we can show that $X^{M \subset G}_{\mu}(b)$ has the same number of orbits of irreducible components as given by the right hand side of Theorem \ref{thm_main}.

 \begin{lemma} \label{lemma sbunionthm}
  $(J_b(F)\cap M(L))\backslash \Sigma'(X_{\mu}^{M\subset G}(b))=W.\mu\cap [\tilde\lambda+(1-\sigma)X_*(T)]$
 \end{lemma}
 \begin{proof}
By Corollary~\ref{cor superbasic} we have
 $$ (J_b(F)\cap M(L))\backslash\bigcup_{\mu'\in I_{\mu,b}}\Sigma^{\rm top}(X_{\mu'}^M(b)) = \bigcup_{\mu'\in I_{\mu,b}}  \big(W_M.\mu'\cap [\tilde\lambda_M+(1-\sigma)X_*(T)]\big).$$
Here the unions on both sides are disjoint, and $\tilde\lambda_M=\tilde\lambda_M(b)$ denotes the element associated with $[b]\in B(M)$ whereas $\tilde\lambda=\tilde\lambda_G(b)$. By Lemma \ref{lemma lambda}, the above union is equal to $\bigcup_{\mu'\in I_{\mu,b}}  \big(W_M.\mu'\cap [\tilde\lambda+(1-\sigma)X_*(T)]\big).$ As $\tilde\lambda$ is minuscule, the set $W_M.\mu'\cap [\tilde\lambda+(1-\sigma)X_*(T)]$ is nonempty for a given $\mu' \in W.\mu$ if and only if $\kappa_M(\mu') = \kappa_M(\tilde\lambda) (= \kappa_M(b))$, i.e.\ iff $\mu' \in I_{\mu,b}$. Hence 
  $\bigcup_{\mu'\in I_{\mu,b}}  \big(W_M.\mu'\cap [\tilde\lambda+(1-\sigma)X_*(T)]\big)= W.\mu\cap [\tilde\lambda+(1-\sigma)X_*(T)].$
 \end{proof}

 In order to relate the irreducible components of $X_\mu^{M\subset G}(b)$ to those of $X_\mu(b)$, we consider the variety 
 \[
  X_\mu^{P\subset G} (b)\coloneqq \{gK_P \in \Grass_P \mid g^{-1}b\sigma(g) \in K\mu(\epsilon)K \}
 \]
 as intermediate object. The inclusion $P \mono G$ induces a natural map $X_{\mu}^{P\subset G}(b)\rightarrow X_{\mu}(b)$. Using the Iwasawa decomposition $G(L)=P(L)K$ we see that this map is surjective, and in fact $X_{\mu}^{P\subset G}(b)$ is nothing but a decomposition of $X_{\mu}^G(b)$ into locally closed subsets (see e.g.\ \cite[Lemma~2.2]{hamacher15a}). Thus we obtain a natural bijection $$\Sigma^{\top}(X_{\mu}^{P\subset G}(b))\rightarrow \Sigma^{\top}(X_{\mu}^G(b))$$ which induces a surjection
 \begin{equation}\label{surj1}
  \alpha_\Sigma\colon  (J_b(F)\cap P(L))\backslash\Sigma^{\top}(X_{\mu}^{P\subset G}(b))\twoheadrightarrow J_b(F)\backslash\Sigma^{\top}(X_{\mu}^G(b)).
 \end{equation}
 Furthermore, $\dim X_{\mu}^{P\subset G}(b)=\dim X_{\mu}^G(b).$

 On the other hand, the restriction of the canonical projection $\Grass_P \epi \Grass_M$ induces a surjective morphism $$\beta\colon  X_{\mu}^{P\subset G}(b)\rightarrow X_\mu^{M\subset G}(b)$$ by \cite[Prop.~2.9]{hamacher15a}.  Moreover the fibre dimension for $x\in X_{\mu'}^M(b)$ is given by
 \begin{equation}\label{fibdim2}
  \dim \beta^{-1}(x)=\dim X_{\mu}^{P\subset G}(b)- \dim X_{\mu'}^{M}(b),
 \end{equation}
 see \cite[Lemma 2.8, Prop. 2.9~(2)]{hamacher15a}, using that for minuscule $\mu$, equality in Lemma 2.8 of loc.~cit.~always holds, and using the dimension formula \cite[Thm.~1.1]{hamacher15a}. Note that this only depends on $\mu'$ (but indeed depends on the choice of $\mu'\in I_{\mu,b}$), but not on the point $x$.

 \begin{lemma}\label{lem61}
  $\beta$ induces a well-defined surjective map $$\beta_{\Sigma}\colon \Sigma^{\top}(X_{\mu}^{P\subset G}(b))\rightarrow \Sigma'(X_{\mu}^{M\subset G}(b)).$$ 
  It is $J_b(F)\cap P(L)$-equivariant for the natural action on the left hand side, and the action through the natural projection $J_b(F)\cap P(L) \epi J_b(F)\cap M(L)$ on the right hand side.
 \end{lemma}
 Recall that a subset of $G(L)$ is called bounded if it is contained in a finite union of $K$-double cosets.
 \begin{proof}
  Let $\mathcal C$ be a top-dimensional irreducible component of $X_{\mu}^{P\subset G}(b)$. Then $\beta(\mathcal{C})$ is irreducible and thus contained in one of the open and closed subschemes $X_{\mu'}^{M}(b).$ By (\ref{fibdim2}), its dimension is equal to $\dim (X_{\mu'}^{M}(b))$, hence $\beta(\mathcal{C})$ is a dense subscheme of one of the irreducible components of $X_{\mu'}^{M}(b)$. In this way we obtain the claimed map $\beta_{\Sigma}$. It is surjective and $J_b(F)\cap P(L)$-equivariant because the same holds for $\beta$.
 \end{proof} 

\begin{proposition}\label{prop56}
 Let $Z \subset X_\mu^{M\subset G}(b)$ be an irreducible subscheme. Then $J_b(F) \cap N(L)$ acts transitively on $\Sigma(\beta^{-1}(Z)).$
\end{proposition}
In the proof we need the following remark.
\begin{remark}\label{remhv24}
For $x\in \widetilde{W}$ let $IxI$ be the locally closed subscheme of $LG$ whose $k$-valued points are $I(k)xI(k)$. Let $\Y$ be a scheme and $g\in (IxI)(\Y)$. Then we claim that there are elements $i_1,i_2\in I(\Y)$ with $g=i_1xi_2$. In equal characteristic, this is \cite[Lemma~2.4]{HV12} (the proof in loc.~cit. shows the above statement, although the Lemma only claims the assertion \'etale locally on $\Y$). Let us explain how to modify the proof to deduce the above statement in general: We consider the morphism $I/(I\cap xIx^{-1})\rightarrow LG/I$ to the affine flag variety given by $g\mapsto gx$. By writing down the obvious inverse one sees that it is an immersion with image $IxI/I$.

Let $g\in (IxI)(\Y)$ and $\overline g$ its image in the affine flag variety. Then the above shows that $\overline g$ is the image of some $\bar i\in I/(I\cap xIx^{-1})(\Y)$. Note that $I/(I\cap xIx^{-1})=I_0/(I_0\cap xI_0x^{-1})$ where $I_0$ is the unipotent radical of $I$. By \cite[Lemma~2.1]{HV12} we can thus lift $\bar i\in I_0/(I_0\cap xI_0x^{-1})(\Y)$ to an element $i_1\in I_0(\Y)$ which is as claimed.
\end{remark}

\begin{proof}[Proof of Proposition \ref{prop56}] As we have to take an inverse image of an element under $\sigma$ later in this proof, we replace all occuring ind-schemes by their perfections. Note that this does not change the underlying topological spaces of the schemes. Moreover, since we may check the assertion on an open covering of $Z$, we may replace $Z$ by an open subscheme $\Y \subset Z$ containing one fixed but arbitrary point $z \in Z(k)$.

 \'Etale locally there is a lifting of the inclusion $Z \mono X_{\mu'}^M(b)$ to $LM$ (\cite[Lemma~1.4]{PR08}, the proof also works for $\cha F = 0$, cf.~\cite[Prop.~1.20]{zhu}). Thus there exists $\Y' \to Z$ \'etale with $z \in \image (\Y' \to Z)$ such that there exists a lift $\iota:\Y' \to LM$. By replacing $\Y'$ by an irreducible component if necessary, we may assume that $\Y'$ is again irreducible. We denote by $\Y$ the image of $\Y'$ in $Z$, and by $y\in \Y'$ a point mapping to $z$.

 We denote
 \[
  \Phi=\{(m,n)\in \iota(\Y')\times N(L) \mid mnK_P\in X_{\mu}^{P\subset G}(b)\}  
 \]
 and $b_m\coloneqq m^{-1}b\sigma(m)$ for any $m \in M(L)$. For $g = mn \in P(L)$ we have
 \begin{equation} \label{eq Frobenius slope filtr}
  g^{-1}b\sigma(g)=b_m\cdot [b_m^{-1}n^{-1}b_m\sigma(n)]
 \end{equation}
 where the bracket is in $N(L)$ and where $b_m\in M(L)$. The condition $gK_P \in \beta^{-1}(\Y)$ is then equivalent to the condition that we may choose $m\cdot n\in gK_P$  with $m\in \iota(\Y')\subset LM$ and $n\in N(L)$ such that the last bracket is in $N(L)\cap b_m^{-1}K\mu(\epsilon) K$. Thus we have a morphism
 \begin{align*}
  \gamma\colon \Phi &\rightarrow \mathcal{E} \coloneqq  \{(m,c)\mid m\in \iota(\Y'), c\in N(L)\cap b_m^{-1}K\mu(\epsilon) K\} \\
  (m,n) &\mapsto (m,b_{m}^{-1}n^{-1}b_{m}\sigma(n)).
 \end{align*}

 In order to get an easier description of $\Ecal$, we show that one can assume $b_m \in K_M\cdot\mu'$ after further shrinking $\Y$ and replacing $\iota$ if necessary. Let $x\in \widetilde W$ such that $I_MxI_M \subset K_M \mu'(\epsilon) K_M$ is the open cell, where $I_M$ denotes the standard Iwahori subgroup of $M$. Then $K\mu(\epsilon) K=KxK$, and we fix $k_0,k_0' \in K$ such that $b_{\iota(y)} = k_0 x k_0'$. We replace $\Y'$ (and thus $\Y$) by the open neighborhood of $y$ such that $b_m \in k_0 \cdot I_M x I_M \cdot k_0'$ for all $m \in \iota(\Y')$.  By Remark \ref{remhv24} we have a global decomposition $b_m=k_0i_1xi_2k_0'$ with $i_j\in I_M(\iota(\Y'))$. As $\Y\subseteq X_{\mu'}^M(b)$ we have $x=w_1\mu' w_2\in W_M\mu W_M$, thus $b_m=k_0i_1w_1\mu'(\epsilon)w_2i_2k_0'$. We now replace $m$ by $m\sigma^{-1}(w_2i_2k_0')^{-1}\in mK_M$ and modify $\iota$ accordingly. With respect to this new choice we obtain a decomposition of $b_m$ of the form  $k_1\mu'(\epsilon)$ with $k_1=\sigma^{-1}(w_2i_2k_0')k_0i_1w_1\in L^+M(\iota(\Y'))$. Now
 \begin{align*}
  N(L)\cap b_m^{-1}K\mu(\epsilon) K&= N(L)\cap \mu'(\epsilon)^{-1}K\mu(\epsilon) K\\
  &= \mu'(\epsilon)^{-1}(N(L)\cdot\mu'(\epsilon)\cap K\mu(\epsilon) K).
 \end{align*}
 Note that this only depends on the constant element $\mu'$. Hence 
 \[
  \mathcal E=\iota(\Y')\times (N(L)\cap \mu'(\epsilon)^{-1}K\mu(\epsilon) K).
 \]
 \noindent{\it Claim 1.} $\mathcal E$ is irreducible. 

 As $\iota(\Y')$ is irreducible, we have to show that $N(L)\cap \mu'(\epsilon)^{-1} K\mu(\epsilon) K$ is irreducible. For this we consider the morphism $\pr_{\mu'}\colon  N(L) \to N(L)\mu'(\epsilon)K \subset \Grass_G, n \mapsto \mu'(\epsilon)n$. Then $N(L)\cap \mu'(\epsilon)^{-1} K\mu(\epsilon) K$ is the preimage of $N(L)\mu'(\epsilon) K \cap K\mu(\epsilon)K$, which is irreducible by \cite[Cor.~13.2]{MV07}. On the other hand $\pr_{\mu'}$ is a $K_N$-torsor, since it is surjective and factorises as
 $$N(L)\rightarrow \Grass_N\hookrightarrow \Grass_G\overset{\mu'(\epsilon)\cdot}{\rightarrow}\Grass_G.$$ Here the first map is the projection, a $K_N$-torsor. The second is the natural closed embedding, and the third the isomorphism obtained by left multiplication by $\mu'(\epsilon)$.
 As $K_N$ is also irreducible, this completes the proof of Claim 1. \smallskip

 \noindent{\it Claim 2.} Let $\mathcal F\subseteq \Phi$ be a non-empty open subscheme with $\mathcal F=\mathcal F K_N$ where $K_N$ acts by right multiplication on the second component. Then its image under $\gamma$ contains an open subscheme of $\Ecal$. In particular, it is dense by Claim 1.

 Fix an irreducible component $C$ of $\Phi $ such that its intersection with $\mathcal F$ is non-empty. We may replace $\mathcal F$ by an open and dense subscheme of points only contained in the one irreducible component $C$. As $\mathcal F$ is invariant under right multiplication by $K_N$ and $m$ is contained in a bounded subscheme of $LM$, its image under $\gamma$ is invariant under right multiplication by some (sufficiently small) open subgroup $K_N'$ of $K_N$ (this follows from the same proof as \cite[Prop. 5.3.1]{GHKR06}, which carries over literally to the unramified case and the case $\cha F = 0$). Thus it is enough to show that the image of $\gamma(\mathcal F)$ in $\iota(\Y')\times (N(L)\cap \mu'(\epsilon)^{-1}K\mu K)/K_N'$ is open. Let $g_0\in \mathcal F$ and let $U=\Spec R$ be an affine open neighborhood of $\gamma(g_0)$ in $\mathcal E$. After possibly replacing $K_N'$ by a smaller open subgroup we may assume that $U$ is $K_N'$-invariant. Let $(m_U,n_U)$ be the universal element. Then $m_U$ and $n_U$ are contained in bounded subsets of $LM$ resp. $LN$. By Corollary \ref{corlift} there is an \'etale covering $R'$ of $R$ and a morphism $\Spec R'\rightarrow\Phi$ such that the composite with $\gamma$ and the quotient modulo $K_N'$ maps $\Spec R'$ surjectively to $U/K_N'$. Intersecting $\Spec R'$ with the inverse image of the open subscheme $\mathcal F$ of $\Phi$ and using that $R\rightarrow R'$ is finite \'etale we obtain an open subscheme of $\Spec R'$, or of $\mathcal F$ mapping surjectively to an open neighbourhood of $g_0K_N'$. This implies the claim. \smallskip

 Finally, we show show that all irreducible components of $\beta^{-1}(\Y)$ are contained in one $J_b(F)\cap N(L)$-orbit of irreducible components of $X_{\mu}^{P\subset G}(b)$. Let $\mathcal D,\mathcal D'$ be irreducible components of $\beta^{-1}(\Y)$. We have to show that all dense open subsets $D,D'$ of the two components contain points $p,p'$ which are in the same $J_b(F)$-orbit. Consider the $K_N$-torsor $$\phi\colon \Phi\rightarrow \beta^{-1}(\Y),\quad (m,n)\mapsto mnK_P.$$ Then it is enough to show that for all non-empty open subsets $C_1,C_2$ of $\Phi$ with $C_iK_N=C_i$ there are points $q_i\in C_i$ and a $j\in J$ with $\phi(q_1)=j\phi(q_2)$. This latter condition follows if we can show that $\gamma(q_1)=\gamma(q_2)$. But by Claim 2, $\gamma(C_1), \gamma(C_2)$ are both open and dense in $\mathcal E$, which implies the existence of such $q_1,q_2$.
\end{proof}

\begin{corollary} \label{cor prosurj2}
 $\beta_{\Sigma}$ induces a bijection $$(J_b(F) \cap P(L)) \backslash \Sigma(X_\mu^{P \subset G}(b)) \bij (J_b(F) \cap M(L)) \backslash \Sigma(X_\mu^{M \subset G}(b))$$ which restricts to $$(J_b(F) \cap P(L)) \backslash \Sigma^{\rm top}(X_\mu^{P \subset G}(b)) \bij (J_b(F) \cap M(L)) \backslash \Sigma'(X_\mu^{M \subset G}(b)).$$ In particular $X_\mu^{P\subset G}(b)$ is equidimensional if and only if the $X_{\mu'}^M(b)$ are for all $\mu' \in I_{\mu,b}$.
\end{corollary}

We use the following notation. Let $R$ be an integral $k$-algebra. In the arithmetic case we assume $R$ to be perfect and let $\mathcal R=W_{\Ocal_F}(R)$. In the function field case, let $\mathcal R=R\pot{\epsilon}$. In both cases let $\mathcal R_L=\mathcal R[1/\epsilon]$.

For $m\in M(\mathcal{R}_L)$ consider the map
\begin{eqnarray*}
f_m\colon  LN_{R}&\rightarrow& LN_R\\
n&\mapsto &(m^{-1}n^{-1}m)\sigma(n).
\end{eqnarray*}

\begin{lemma}[Chen, Kisin, Viehmann]\label{lemlift}
Let $b\in [b]\cap M(L)$ with $b\sigma(b)\dotsm \sigma^{l_0-1}(b)=\epsilon^{l_0\nu_b}$ for some $l_0>0$ such that $l_0\nu_b\in X_*(T)$. Let $R$ be an integral $k$-algebra, $\mathcal R, \mathcal{R}_L$ as above and $y\in N(\mathcal{R}_L)$ contained in a bounded subscheme. Let further $x_1\in \Spec R(k)$ and $z_1\in N(L)$ with $f_b(z_1)=y(x_1)$. Then for any bounded open subgroup $K'\subset N(L)$ there exists a finite \'etale covering $R\rightarrow R'$ with associated $\mathcal R\rightarrow\mathcal R'$ and $z\in N(\mathcal{R}'_L)$ such that 
\begin{enumerate}
\item for every $k$-valued point $x$ of $R'$ we have $f_b(z(x))K'=y(x)K'$
\item there exists a point $x_1'\in \Spec R'(k)$ over $x_1$ such that $z(x_1')=z_1$.
\end{enumerate}
\end{lemma}
\begin{proof}
This is \cite[Lemma 3.4.4]{CKV15}, except for the fact that in loc.~cit., $R$ is assumed to be smooth, and only the case of mixed characteristic is considered. But actually, none of these assumptions is needed in the proof given there.
\end{proof}
\begin{corollary}\label{corlift}
Let $b\in [b]\cap M(L)$, $R$ and $\mathcal R$ be as in the previous lemma. Let $m\in M(\mathcal{R}_L)$, and $y\in N(\mathcal{R}_L)$, each contained in a bounded subscheme. Let further $x_1\in \Spec R(k)$ and $z_1\in N(L)$ with $f_b(z_1)=y(x_1)$. Let $b_m=m^{-1}b\sigma(m)\in M(\mathcal{R}_L) $. Then for any bounded open subgroup $K'\subset N(L)$ there exists a finite \'etale covering $R\rightarrow R'$ with associated extension $\mathcal R\rightarrow\mathcal R'$ and $z\in N(\mathcal{R}'_L)$ such that 
\begin{enumerate}
\item for every $k$-valued point $x$ of $R'$ we have $f_{b_m}(z(x))K'=y(x)K'$
\item there exists a point $x_1'\in \Spec R'(k)$ over $x_1$ such that $z(x_1')=z_1$.
\end{enumerate}
\end{corollary}
\begin{proof}
For $n\in N(L)$ we have
\begin{eqnarray*}
f_{b_m}(n)&=&(\sigma(m)^{-1}b^{-1}m)n^{-1}(m^{-1}b\sigma(m))\sigma(n)\\
&=& \sigma(m)^{-1}b^{-1}(mn^{-1}m^{-1})b\sigma(mnm^{-1})\sigma(m)\\
&=&\sigma(m)^{-1}f_b(mnm^{-1})\sigma(m).
\end{eqnarray*}
By the boundedness assumption on $m$, there is a bounded open subgroup $K''$ such that $\sigma(m(x))^{-1}K''\sigma(m(x))\in K'$ for all $\overline k$-valued points $x$ of $\Spec R$. Applying Lemma \ref{lemlift} to $\sigma(m)y\sigma(m)^{-1}$ and $K''$, and conjugating the result by $m$, we obtain the desired lifting with respect to $f_{b_m}$.
\end{proof}
\begin{theorem} \label{thmnewmain}
 Let $\mu\in X_*(T)_{\dom}$ be minuscule, $b\in [b] \in B(G,\mu)$, and $\tilde\lambda\in X_\ast(T)$ an associated element. Then the map 
 \[
  \phi=\alpha_{\Sigma}\circ \beta_{\Sigma}^{-1}\colon   W.\mu \cap [\tilde\lambda+(1-\sigma)X_*(T)] \rightarrow J_b(F)\backslash \Sigma^\top(X_{\mu}(b))
 \]
constructed above is surjective and it is bijective if and only if $J_b(F)$ acts trivially on $(J_b(F) \cap P(L)) \backslash \Sigma^{\top}(X_\mu(b)).$
\end{theorem}
\begin{proof}
From Lemma \ref{lemma sbunionthm}, Corollary~\ref{cor prosurj2}, and \eqref{surj1} we obtain the claimed maps
 \begin{eqnarray*}
   W.\mu\cap [\tilde\lambda+(1-\sigma)X_*(T)] &=& (J_b(F)\cap M(L))\backslash \Sigma(X_\mu^{M\subset G}(b))  \\
   &\overset{\beta_\Sigma^{-1}}{\rightarrow} & (J_b(F) \cap P(L)) \backslash \Sigma^{\rm top}(X_\mu^{P\subset G}(b)) \\
   &\overset{\alpha_\Sigma}{\twoheadrightarrow}& J_b(F) \backslash \Sigma^{\top}(X_\mu(b)).
  \end{eqnarray*}
As $\Sigma^{\top}(X_\mu(b))\cong \Sigma^{\rm top}(X_\mu^{P\subset G}(b))$, this description also implies the assertion about bijectivity.
\end{proof}

\begin{proof}[Proof of Theorem \ref{thm_main}]
The first assertion is a direct consequence of the previous theorem.

If $G$ is split, then $W.\mu\cap [\tilde\lambda+(1-\sigma)X_*(T)]=\{\tilde\lambda\}$ has only one element, hence the map is also injective.

If the second condition holds, then $J_b(F)\subset P(L)$, hence $\alpha_{\Sigma}$ and also $\phi$ are bijective.
\end{proof}

\end{document}